\documentclass[12pt,a4paper,oneside]{article}
\usepackage{amsmath,amssymb,amsthm}
\usepackage{amsfonts}
\usepackage{amssymb}
\usepackage{graphicx}
\usepackage{amsthm}
\usepackage{newlfont}
\newlength{\defbaselineskip}
\newcommand{\setlinespacing}[1]%
           {\setlength{\baselineskip}{#1 \defbaselineskip}}

\theoremstyle{plain}
\newtheorem{theorem}{Theorem}[section]

\theoremstyle{definition}
\newtheorem{definition}[theorem]{Definition}
\newtheorem{remarks}[theorem]{Remark}
\newtheorem{example}[theorem]{Example}

\begin{document}

\begin{center}
\vspace{3cm}
 {\bf \large Fibonacci Fervour in Linear Algebra and Quantum Information Theory}\\
\vspace{3cm}

{\bf Manami Chatterjee$^*$, Ajit Iqbal Singh$^{**}$ and K.C. Sivakumar$^*$}\\
\vspace{.25cm}
$^*$Department of Mathematics\\
Indian Institute of Technology Madras\\
Chennai 600 036, India. \\
\vspace{.25cm}
$^{**}$INSA Emeritus Scientist\\
The Indian National Science Academy \\
New Delhi 110 002, India.\\
\end{center}

\begin{abstract}
This is a survey on certain results which bring about a connection between Fibonacci sequences on the one hand and the areas of matrix theory and quantum information theory, on the other. 
\end{abstract}

\newpage
\section{Introduction}
As evidenced in the mathematics literature, several mathematical problems eventually reduce to studying one sequence of numbers or the other. In this regard, the Fibonacci sequence is one of the most well known sequences of positive integers that are characterized by the fact that every term after the first two terms is the sum of the two preceding ones, given the first two terms. The name Fibonacci comes after the Italian mathematician Leonardo of Pisa, known as Fibonacci (1175 - 1250). For a natural number $n$, the $n$th Fibonacci number is denoted by $F_n$ and they are defined recursively by $F_1=F_2=1$ and $F_n=F_{n-1}+F_{n-2}$ for $n\geq 3$. (Later, in many cases even $F_0$ is being used with value zero).

It must be pointed out that there are many journals in which there is a lot of literature exploring Fibonacci numbers and their different aspects. However, a full journal by the name ``The Fibonacci Quarterly'' has been completely devoted to their study. It is noteworthy that the Fibonacci sequence occurs very frequently in mathematics and many of their interesting properties have been established. Let us recall that the applications of Fibonacci sequences include computer algorithms like the Fibonacci search technique and the Fibonacci heap data structure. Graphs called Fibonacci cubes are used in parallel and distributed systems. It is amusing and amazing at the same time, that Fibonacci numbers are used to explain certain natural biological structures called phyllotaxis, the sprouts of some fruits like pineapple, an uncurling fern and in the manner in which branches appear in trees. 

In the literature, there is an abundance of results on Fibonacci sequences and that makes a complete short survey virtually impossible. Our interest in this survey article is to study certain interesting connections between Fibonacci numbers, their relationships with certain aspects in the area of Matrix Theory and some applications, particularly in Quantum Information Theory.

Let us start with an interesting illustration on how Fibonacci numbers and matrices are related. Note that one has 
\begin{eqnarray}
det\left( \begin{array}{cc}
F_n & F_{n-1}\\
F_{n+1} & F_n
\end{array} \right)&=&det\left( \begin{array}{cc}
F_n & F_{n-1}\\
F_{n+1}-F_n & F_n-F_{n-1}
\end{array} \right)\nonumber\\
&=&det\left( \begin{array}{cc}
F_n & F_{n-1}\\
F_{n-1} & F_{n-2}
\end{array} \right)\nonumber\\
&=&-det\left( \begin{array}{cc}
F_{n-1} & F_{n-2}\\
F_n & F_{n-1}
\end{array} \right).\nonumber
\end{eqnarray}
Thus, $F_n^2-F_{n-1}F_{n+1}=F_nF_{n-2}-F_{n-1}^2.$ Proceeding in this manner inductively, it follows that $F_n^2-F_{n-1}F_{n+1}=(-1)^{n+1}$.

As another instance, starting with the symmetric matrix $\left( \begin{array}{cc}
1 & 1\\
1 & 0
\end{array} \right)$ and by computing its eigenvalues $\lambda=\frac{1+\sqrt{5}}{2}$ and $\mu=\frac{1-\sqrt{5}}{2}$ and a corresponding orthogonal basis of eigenvectors $\bigg\{\left( \begin{array}{c}
\lambda \\
1
\end{array} \right), \left( \begin{array}{c}
\mu \\
1
\end{array} \right)\bigg\}$, one may determine the exact value of the $n$th Fibonacci number using the equation 
\begin{equation}
\left( \begin{array}{c}
F_n \\
F_{n-1}
\end{array} \right)= \left( \begin{array}{cc}
1 & 1\\
1 & 0
\end{array} \right)\left( \begin{array}{c}
F_{n-1} \\
F_{n-2}
\end{array} \right),\nonumber
\end{equation}
and by proceeding by the principle of induction to show that 
\begin{equation}
\left( \begin{array}{c}
F_n \\
F_{n-1}
\end{array} \right)=\left( \begin{array}{cc}
1 & 1\\
1 & 0
\end{array} \right)^{n-1}\left( \begin{array}{c}
F_1 \\
F_0 
\end{array} \right)=\left( \begin{array}{cc}
1 & 1\\
1 & 0
\end{array} \right)^{n-1}\left( \begin{array}{c}
1 \\
0 
\end{array} \right).\nonumber
\end{equation}
Using the above, one can obtain a closed form expression of the $n$th Fibonacci number, as 
\begin{center}
$F_n=\frac{(1+\sqrt{5})^n-(1-\sqrt{5})^n}{2^n \sqrt{5}}.$
\end{center}
The technique can be modified to be applicable to any recursion formula like Fibonacci numbers, say given by a matrix of the form $\left( \begin{array}{cc}
a & b\\
1 & 0
\end{array} \right)$, with $a, b >0$ or using different initial column vectors. We will have occasion to use this latter, in the last section.

Here is an outline of the contents of this survey, which may be considered to have mainly two overarching objectives. The first purpose is to present an overview of some hand picked results in Linear Algebra which have Fibonacci fervour. This task is undertaken in Section \ref{pmfne}, which in turn has various subsections focussing on specific topics, viz., general matrices with entries given by Fibonacci numbers in subsections \ref{Fibo} and \ref{kfibo} and, the case of circulant matrices in subsection \ref{circfibo}. While subsection \ref{dfibo} briefly surveys special matrices whose determinants are Fibonacci numbers, the last subsection of Section \ref{pmfne} recalls some rather recent results on certain interesting relationships between sums of entries of $\{0,1\}$-matrices and Fibonacci numbers. The second aim of this survey is to present an exposition of essentially fundamental ideas in Quantum Information Theory that have certain genuine connections with Fibonacci and Lucas number sequences. This is presented in Section \ref{fibqit}. Starting with a detailed discussion for the case of two dimensions as motivation in subsection \ref{twodim}, we move on to higher dimensions in subsection \ref{highdim}. The concept of symmetric informationally complete positive operator valued measure (SIC-POVM) is the focus in subsection 3.3
, whereas subsection \ref{sicant} brings in its relevance to algebraic number theory. The penultimate subsection \ref{flsic} deals with certain particular results between SIC-POVM and the number sequences of Fibonacci and Lucas, completing the circle of discussion. 

\section{Properties of Matrices with Fibonacci Numbers as its Entries}\label{pmfne}
In this section, we present a review of the various properties of matrices whose entries are either Fibonacci numbers or their variants.

\subsection{Fibonacci Matrices}\label{Fibo}
The name {\it Fibonacci matrix} is used in a variety of contexts in the literature. For instance, Lee, Kim and Lee \cite{LKL} proposed the following definition: Let $\mathbb{F}_n(=[f_{ij}])$ be the $n\times n$ matrix whose entries are given by
\begin{center}
	$f_{ij}=\left\{\begin{array}{cc}
	F_{i-j+1}, & i-j+1> 0\\
	0, & i-j+1\leq 0
	\end{array} \right.$
\end{center}
They studied certain factorizations of these and their symmetric versions. Here, the $n\times n$ {\it symmetric Fibonacci matrix} $\mathbb{Q}_n=[q_{ij}]$ is defined as follows: Let $q_{i0}$ be set to zero. Next let,
\begin{center}
	$q_{ij}=q_{ji}=\left\{\begin{array}{cc}
	\sum_{k=1}^{i}F_k^2, & i=j\\
	q_{i,j-2}+q_{i,j-1}, & i+1\leq j
	\end{array}\right.$
\end{center}
They showed that $\mathbb{Q}_n$ has the Cholesky factorization given by $\mathbb{Q}_n=\mathbb{F}_n \mathbb{F}_n^T$.  

\subsection{$k$-Fibonacci Matrices}\label{kfibo}
Lee and Kim extended the concepts given in \ref{Fibo} to $k$-Fibonacci matrices and $k$-symmetric Fibonacci matrices \cite{LK}. For a positive $k\geq 2$, the {\it $k$-Fibonacci sequence} $\{F(k)_n\}$ is defined as: 
\begin{center}
$F(k)_1=F(k)_2=\cdots =F(k)_{k-2}=0$, $F(k)_{k-1}=F(k)_k=1$
\end{center}
and for $n> k (\geq 2)$, 
\begin{center}
$F(k)_n=F(k)_{n-1}+F(k)_{n-2}+\cdots +F(k)_{n-k}$.
\end{center}
Then an $n\times n$ {\it $k$-Fibonacci matrix} $\mathbb{F}(k)_n=[f(k)_{ij}]_n$ is defined for a fixed $k\geq 2$ via,
\begin{center}
	$f(k)_{ij}=\left\{\begin{array}{cc}
	F(k)_{i-j+1+(k-2)}, & i-j+1> 0\\
	0, & i-j+1\leq 0
	\end{array} \right.$
\end{center} 
Also the $n\times n$ {\it $k$-symmetric Fibonacci matrix} $\mathbb{Q}(k)_n=[q(k)_{ij}]_n$ is defined via,
\begin{center}
	$q(k)_{ij}=\left\{\begin{array}{cc}
	\sum_{l=1}^{k}q(k)_{i,j-l}, & i+1\leq j\\
	\sum_{l=1}^{k}q(k)_{i,i-l}+F(k)_{k-1}, & i=j,
	\end{array} \right.$
\end{center}
where $q(k)_{ij}=0$ for $j\leq 0$. Then $\mathbb{F}(2)_n$ and $\mathbb{Q}(2)_n$ reduce to $\mathbb{F}_n$ and $\mathbb{Q}_n$ described earlier in subsection \ref{Fibo}. The Cholesky factorization of $\mathbb{Q}(k)_n$ is given by $\mathbb{Q}(k)_n=\mathbb{F}(k)_n \mathbb{F}(k)_n^T$. Some factorizations of $\mathbb{F}(k)_n$ are also given in \cite{LK}.
		
Several other matrices have been studied in the literature like the Bell matrix, the Pascal matrix and Stirling matrices of various types, with which Fibonacci matrices are related. Lee, Kim, Cho \cite{LKC} gave factorizations of the Pascal matrix and Stirling matrices, in terms of Fibonacci matrices. Also, they discussed some identities involving Fibonacci numbers and binomial coefficients. Zhang and Wang \cite{ZW} revealed a factorization for the Pascal matrix involving Fibonacci matrices. Subsequently, Wang and Wang \cite{WW} obtained a factorization with one factor as a Fibonacci matrix, for the Bell matrix.

\subsection{Circulant Matrices Whose Entries are Fibonacci Numbers}\label{circfibo}
There are notions of circulant matrices (see \cite{Geller}, \cite{Kra}) with Fibonacci numbers as their entries. An $n\times n$ {\it right circulant matrix} $C=Circ(c_0, c_1, \cdots , c_{n-1})$ is of the form 
\begin{center}
$C=\left(\begin{array}{ccccc}
c_0 & c_1 & c_2 & \cdots & c_{n-1}\\
c_{n-1} & c_{0} & c_1 & \cdots & c_{n-2}\\
c_{n-2} & c_{n-1} & c_0 & \ddots & c_{n-3} \\
\colon & \colon & \ddots & \ddots & \colon\\
c_1 & c_2 & c_3 & \cdots & c_0
\end{array} \right). $
\end{center}
Observe that each row is a cyclic shift of the previous row to the right. If such a shift is to the left then we obtain a {\it left circulant matrix} and it is denoted by $LCirc(c_0, c_1, \cdots, c_{n-1})$. The following theorem tells us about the invertibility of a right circulant matrix:

\begin{theorem}\label{circinv}
Let $V_n= Circ(v_0, v_1,\cdots v_{n-1})$. Then we have the following:\\
$(a)$ Let $f(x)= \sum _{j=0}^{n-1} v_jx^j$ and $\omega = exp(2\pi i/ n)$. Then $V_n$ is invertible if, and only if, $f(\omega ^k)\neq 0$ for all $k=0, 1, 2, \cdots, n-1$.\\
$(b)$ If $V_n$ is invertible, then its inverse is also a right circulant matrix.
\end{theorem}

Continuing with our discussion on circulant matrices, Lind gave a determinant formula for $ Circ(F_r, F_{r+1}, \cdots, F_{r+n-1})$ $(r\geq 1)$ \cite{Lind}. Later, using Theorem \ref{circinv} and other tools,  Shen, Cen and Hao considered circulant matrices with Fibonacci and Lucas numbers as entries \cite{Shen}. They obtained the value of their determinants and computed their inverses, explicitly. 
One must mention that {\it Lucas numbers} are closely related to Fibonacci numbers and defined as $L_n=L_{n-1}+L_{n-2}$, $n\geq 3$ with $L_1=1$ and $L_2=3$ as the two initial terms.

 Altin\c sik, Yal\c cin and B\"uy\"ukk\"ose studied $Circ(F_1^*, F_2^*, \cdots, F_n^*)$ and showed that such matrices are invertible \cite{Altin}. Here, $F_n^*$ stands for the {\it complex Fibonacci number} defined as $F_n^*= F_n+iF_{n+1}$. They also computed their determinants and inverses. Jiang, Gong and Gao introduced two new sequences in terms of the sum and product of Fibonacci and Lucas numbers \cite{Gong}. For right circulant and left circulant matrices whose entries are given by these sums, they determined the inverses and determinants. For similar types of results involving $k$-Fibonacci and $k$-Lucas numbers, we refer the reader to the work by Jiang, Gong and Gao \cite{kFibo}.

Simultaneously, two other notions namely, skew circulant matrices and skew left circulant matrices have been considered in the literature. A {\it skew circulant matrix} and a {\it left skew circulant matrix} with first row $(c_0, c_1, \cdots , c_{n-1})$ are defined as 
\begin{center}
$\left(\begin{array}{ccccc}
	c_0 & c_1 & c_2 & \cdots & c_{n-1}\\
	-c_{n-1} & c_{0} & c_1 & \cdots & c_{n-2}\\
	-c_{n-2} & -c_{n-1} & c_0 & \ddots & c_{n-3} \\
	\colon & \colon & \ddots & \ddots & c_1\\
	-c_1 & -c_2 & -c_3 & \cdots & c_0
	\end{array} \right)$
\end{center}
and 
\begin{center}
$\left(\begin{array}{cccccc}
	c_0 & c_1 & \cdots & c_{n-3} & c_{n-2} & c_{n-1}\\
	c_1 & c_2 & \cdots & c_{n-2} & c_{n-1} & -c_0\\
	c_2 & c_3 & \cdots & c_{n-1} & -c_0 & -c_1 \\
	\vdots & \vdots & \ddots & \ddots & \vdots & \vdots\\
	c_{n-2} & c_{n-1} & -c_0 & \cdots & -c_{n-4} & -c_{n-3}\\
	c_{n-1} & -c_0 & -c_1 & \cdots & -c_{n-3} & -c_{n-2}
	\end{array} \right)$
\end{center}
respectively. Gao, Jiang and Gong calculated the determinants and inverses of such skew circulant matrices involving Fibonacci and Lucas numbers \cite{Gao}. Jiang, Yao and Lu studied the matrices $SCirc(F_{r+1}, F_{r+2},\cdots, F_{r+n})$ and $SLCirc(F_{r+1}, F_{r+2},\cdots, F_{r+n})$ \cite{Jian}. They presented explicit determinants and inverses of these special matrices which reduce to the formulae of \cite{Gao} for $r=0$. Again, matrices whose entries are defined as the sum of Fibonacci and Lucas numbers were considered, this time in the context of skew circulant and skew left circulant matrices by Jiang and Wei in the work \cite{Zhao}. Formulae for determinants and inverses of such matrices were presented. 
	
Karaduman \cite{Erdal}, defined $k$-sequences of the generalized order-$k$-Fibonacci numbers and studied the determinants of matrices consisting of these numbers. We refer the reader to the work by Fu and Zhou \cite{Fu} for the definition of general order-$k$ sequence and determinants of special matrices whose entries have been taken from these general order-$k$ sequences. Most of the results of \cite{Erdal} come as a special case in \cite{Fu}. In the work of Tasci and Kilic, reported in \cite{Tasci}, the authors defined the order-$k$ generalized Lucas numbers.  These are a particular case of the numbers defined in \cite{Erdal}. They investigated some interesting properties of such sequences. We also would like to mention the work of Karaduman \cite{Erdal1} for properties of the determinants of matrices obtained by generalized order-$k$ Fibonacci numbers.

\subsection{Fibonacci Numbers as Determinants of Certain Special Matrices}\label{dfibo}
Fibonacci numbers are related to special forms of matrices like Hessenberg matrices, triangular matries, tridiagonal matrices etc. via their determinants, sum of the entries etc. An $n \times n$ matrix $A=(a_{ij})$ is an upper(lower) triangular matrix if $a_{ij}=0$ when $i>j$ ($j>i$) and upper(lower) Hessenberg matrix if $a_{ij}=0$ when $i+1>j$ ($j+1>i$). Ching showed that for the collection of all $n\times n$ lower Hessenberg matrices with entries $0$ and $1$, the maximum determinant is $F_n$ \cite{Li}. Strang \cite{Strang} presents a family of tridiagonal matrices given by
	\begin{center}
	$M(n)=\left(\begin{array}{ccccc}
	3 & 1 & 0 & \cdots & 0\\
	1 & 3 & 1 & \cdots & 0\\
	0 & 1 & 3 & \ddots & 0\\
	\vdots & \vdots & \ddots & \ddots & \vdots\\
	0 & 0 & 0 & \cdots & 3
	\end{array} \right)_{n\times n}.$
	\end{center}
It is shown by induction that the determinant of $M(n)$ is the Fibonacci number $F_{2n+2}$. Another example is the family of tridiagonal matrices given by
	\begin{center}
	$H(n)=\left(\begin{array}{ccccc}
	1 & i & 0 & \cdots & 0\\
	i & 1 & i & \cdots & 0\\
	0 & i & 1 & \ddots & 0\\
	\vdots & \vdots & \ddots & \ddots & \vdots\\
	0 & 0 & 0 & \cdots & 1
	\end{array} \right)_{n\times n}$
	\end{center}
where $i$ is the imaginary unit, $i=\sqrt{-1}$. Cahill, Errico, Narayan and Narayan showed that the determinant of $H(n)$ is $F_{n+1}$ \cite{College}. In that work, the authors presented a recurence relation for the determinants of a general lower Hessenberg matrix and derived various other examples of tridiagonal matrices whose determinants are Lucas numbers, odd Fibonacci numbers $F_1, F_3, F_5, \cdots$ and even Fibonacci numbers $F_2, F_4, F_6, \cdots$. If the off-diagonal entries of $H(n)$ are replaced by $1$ (above diagonal) and $-1$ (below diagonal) then it has been shown that the determinant remains the same ($F_{n+1}$) \cite{Strang}. Cahill and Narayan \cite{Naru}, defined a  symmetric family of tridiagonal matrices $M_{\alpha,\beta}(k)$ and $T_{\alpha,\beta}(k)$ with $\alpha$, $\beta$ being positive integers, where $k$ is the order of the matrices. They showed that $det(M_{\alpha, \beta}(k))=F_{\alpha k+\beta}$ and $det(T_{\alpha, \beta}(k))=L_{\alpha k+\beta},$ i.e., the determinants form subsequences of Fibonacci and Lucas numbers. For related results we refer to the works \cite{Nalli}, \cite{Pavel} and \cite{Feng}.

\subsection{Some Recent Results and their Extensions}	
Now, let us recall some recent results. Let $S(X)$ denote the sum of the entries of a matrix $X$. Huang, Tam and Wu \cite{Huang} showed that a number $s$ is equal to $S(A^{-1})$ for an adjacency matrix $A$ (a symmetric $(0,1)$ matrix with trace zero) if, and only if, $s$ is rational. Motivated by this work, Farber and Berman \cite{Far} presented a nice connection between Fibonacci numbers and matrix theory, thereby providing a partial answer to the question "what can be said about the sum of the entries of the inverse of a $(0,1)$ matrix?". Their result may be stated as: A number $s$ is the sum of the entries of the inverse of an $n \times n$ upper triangular matrix $(n \geq 3)$ with entries from the set $ \lbrace 0, 1 \rbrace$ if, and only if, $s$ is an integer between $2-F_{n-1}$ and $2+F_{n-1}$.
	
An extension of this result for group invertible matrices is presently being investigated by the first and third authors \cite{mankcs}. Let us present a brief account of such a generalization. The details will appear elsewhere. For every $n\geq 6$, let $C_1,C_2, C_3$ and $C_4$ be nonsingular matrices of order $n-1$ whose entries are from the set $\{0,1\}$. Let $p_n=S(C_1^{-2}u^n), r_n=S(C_2^{-2}v^n), q_n=-S(C_3^{-2}w^n)$ and $s_n=-S(C_4^{-2}z^n)$, where $u^n, v^n, w^n$ and $z^n$ are vectors with $n-1$ coordinates, whose entries are either $0$ or $1$. It is shown that for certain specific choices of these vectors, the numbers $p_n ,q_n, r_n$ and $s_n$ are nonnegative integers. Then the following result is shown: 

\begin{theorem}\label{grinvfib}
Let $s$ be an integer satisfying either:
\begin{eqnarray}
2-F_{n-2}-q_n \leq s \leq 2+F_{n-2}+p_n,  \nonumber
\end{eqnarray}
or
\begin{eqnarray}
2-F_{n-2}-s_n \leq s \leq 2+F_{n-2}+r_n.  \nonumber
\end{eqnarray}
Then there exists an upper triangular, $\lbrace0,1\rbrace$, group invertible, singular matrix $A$ such that $S(A^\#)=s$.
\end{theorem}

\begin{remarks}
Examples exist to show that the converse of Theorem \ref{grinvfib} does not hold. 
\end{remarks}

\section{Fibonacci and other Relevant Number Sequences in Quantum Tomography}\label{fibqit}
In this section, we shall give an exposition on the relationships between Fibonacci sequences and Quantum Tomography. While our discussion in the earlier part of this survey concerned mainly matrices with real entries (with some exceptions, viz., in the case of circulant matrices) here, we are interested in vectors and matrices in the set-up of finite dimensional linear spaces mostly over the field $\mathbb{C}$ of complex numbers. We shall make no attempt at the history or the development of the topic under consideration. On the other hand, our objective is to create a feel for Quantum Tomography, in as simple a way as one possibly could. We refer to essentially general articles like \cite{FuchHuang}, \cite{FuchStacy}, and references therein, particularly the originators \cite{Zauner}, \cite{Renes}, \cite{Lem}. We mainly draw upon \cite{Bengtsson}, \cite{ScottGrassl}, \cite{NewaddedApple}, \cite{Applebytower} and \cite{Applebyray} and use them freely in the form suited for our purpose.

\subsection{The Two Dimensional Case as a Motivation} \label{twodim}
We begin with the usual inner product space $\mathbb{C}^2$ with the standard basis consisting of $\vert 0 \rangle =e_0=\left( \begin{array}{c}
1 \\
0
\end{array} \right)$ and  $\vert 1 \rangle =e_1=\left( \begin{array}{c}
0 \\
1
\end{array} \right)$.

The space of linear operators on $\mathbb{C}^2$ to itself can be identified with that of $2 \times 2$ complex matrices, viz., $M_2(\mathbb{C})$, or in short, $M_2$; $M_2$ can be made into an inner product space via $\langle T, S \rangle = $ trace of $T^*S$, in short, $tr(T^*S)$. Here $T^*$ denotes the adjoint of $T$ and we follow the convention in Physics or in Quantum Information Theory, that the inner product is linear in the second variable but conjugate-linear in the first variable.

The matrices 
\begin{center}
$U_0=P=\left( \begin{array}{cc}
1 & 0\\
0 & 1
\end{array} \right),$ (the identity matrix)
\end{center}
\begin{center}
$U_1=A=\left( \begin{array}{cc}
1 & 0\\
0 & -1
\end{array} \right)$,
\end{center}
\begin{center}
$U_2=U=\left( \begin{array}{cc}
0 & 1\\
1 & 0
\end{array} \right)$
\end{center}
and 
\begin{center}
$U_3=L=\left( \begin{array}{cc}
0 & -i\\
i & 0
\end{array} \right)$ 
\end{center}
are hermitian and unitary. Further, $AU=iL,~UL=iA$ and $LA=iU$. The matrices $P, A, U, L$ are called Pauli matrices in honour of W.E. Pauli (1900-1958) who was awarded the Nobel prize in physics in 1945. Observe that $\lbrace \frac{1}{\sqrt{2}}U_j,~j=0,1,2,3\rbrace$ is an orthonormal basis for $M_2$, due to the fact that $tr(U_j^*U_k)=2\delta_{jk}$ for $0 \leq j, k \leq 3$. Also, any matrix $T\in M_2$ has the form
\begin{center}
$T=pP+aA+uU+lL=\left( \begin{array}{cc}
p+a & u-li\\
il+u & -a+p
\end{array} \right),$
\end{center}
with $p, a, u, l \in \mathbb{C}$. We note that $T$ is hermitian if, and only if, $p, a, u, l$ are all real. Further in this case, $T$ is positive semi-definite if, and only if, $p\geq 0$ and $p^2 \geq a^2+u^2+l^2$. It is for this reason that the set of all positive semi-definite (in short, positive) matrices is referred to as a {\it cone}. The numbers $p, a, u, l$ are simply $\frac{1}{2}tr(TP),~\frac{1}{2}tr(TA), ~\frac{1}{2} tr(TU)$ and $\frac{1}{2}tr(TL),$ respectively.

Let $\xi$ be a unit vector in $\mathbb{C}^2$. In what follows, we let $P_\xi$ denote the rank-one projection on $\mathbb{C}^2$ given by $P_\xi(x)=\langle \xi, x\rangle \xi$, $x \in \mathbb{C}^2$. Then, for $T \in M_2$, $tr(TP_\xi)=\langle \xi,T \xi \rangle$. So, we look for a basis for $M_2$ consisting of projections like $P_\xi$, viz., (self-adjoint) rank-one projections.

\begin{example}\label{anotherbasis}
One instance for such a sought after basis is provided by $\xi^{(0)}=\vert 0 \rangle=\left( \begin{array}{c}
1 \\
0
\end{array} \right)$, $\xi^{(1)}=\vert 1 \rangle=\left( \begin{array}{c}
0 \\
1
\end{array} \right)$, both unit eigenvectors for $A$, $\xi^{(2)}=\frac{1}{\sqrt{2}}(\vert 0 \rangle +\vert 1 \rangle)=\frac{1}{\sqrt{2}}\left( \begin{array}{c}
1 \\
1
\end{array} \right)$ and $\xi^{(3)}=\frac{1}{\sqrt{2}}(\vert 0 \rangle +i\vert 1 \rangle)=\frac{1}{\sqrt{2}}\left( \begin{array}{c}
1 \\
i
\end{array} \right)$ which are eigenvectors for $U$ and $L,$ respectively for the eigenvalue $1$; and then consider 
$\lbrace P_{\xi^{(j)}}~:~0\leq j \leq 3\rbrace $. This works fine due to the reason that 
\begin{center}
$U_0=P=I=P_{\xi^{(0)}}+P_{\xi^{(1)}}$,                                                                                                                   
\end{center}
\begin{center}
$U_1=A=P_{\xi^{(0)}}-P_{\xi^{(1)}}$,
\end{center}
\begin{center}
$U_2=U=P_{\xi^{(2)}}-(P_{\xi^{(0)}}+P_{\xi^{(1)}}-P_{\xi^{(2)}})$,
\end{center}
\begin{center}
$U_3=L=P_{\xi^{(3)}}-(P_{\xi^{(0)}}+P_{\xi^{(1)}}-P_{\xi^{(3)}})$
\end{center}
and the earlier proved fact that $\lbrace U_j:0\leq j \leq 3\rbrace$ is a basis for $M_2$.
\end{example}

\begin{remarks}\label{rankoneproj}
The following observation will be useful in the sequel. For unit vectors $\xi,~\eta$ in $\mathbb{C}^2$, we take $T=P_\eta$ and obtain $tr(P_\eta P_\xi)=\vert\langle \xi, \eta \rangle \vert ^2$; in other words, $\langle P_\xi, P_\eta \rangle =\vert\langle \xi, \eta \rangle \vert ^2$.
\end{remarks}

For any subset $S$, let $\left\vert{S}\right\vert$ denote the cardinality of the set $S$.

\begin{definition}\label{defequiortho}
Let $S$ be a subset of $\mathbb{C}^2$ consisting of unit vectors satisfying the following conditions:\\
$(a)~ \left\vert{S}\right\vert \geq 3.$\\
$(b)$ For $\xi$, $\eta \in S$, one has $\xi=\lambda \eta$ for some $\lambda \in \mathbb{C}$ if and only if $\xi=\eta$ i.e., if $T \subset S$ with $\left\vert{T}\right\vert=2$, then $T$ is linearly independent. \\
$(i)$ Then $S$ will be called an {\it admissible set}.\\
$(ii)~S$ will be called {\it equiangular} if, and only if, $\lbrace \vert\langle \xi, \eta \rangle\vert~ :~\xi, \eta \in S, \xi \neq \eta \rbrace$ is a singleton, say $\{a\}$. 
\end{definition}

In this case, in view of condition $(b)$, $a\neq 1$ and in view of condition $(a)$ and Remark \ref{rankoneproj}, one has $a\neq 0$. This yields $0<a_s=cos^{-1}a<\frac{\pi}{2}$ and we call $a_s$ the {\it common angle} for $S$.

\begin{remarks}\label{admissrem}
Let $S$ be an admissible set.\\
(a) For any function $f:S\rightarrow [0, 2\pi)$, the subset of $\mathbb{C}^2$ given by $S^f=\lbrace e^{if(\xi)}\xi :~\xi \in S\rbrace,$ is admissible. Also, $\vert S \vert= \vert S^f \vert$. \\
(b)$S$ is equiangular if, and only if, $S^f$ is equiangular. In this case, the common angle is the same for both $S$ and $S^f$. We treat all such sets $S^f$ to be equivalent.
\end{remarks}

There are uncountably many orthonormal basis in $\mathbb{C}^2$. On the other hand, any non-empty orthonormal set in $\mathbb{C}^2$ has cardinality $1$ or $2$. So, by Remark \ref{rankoneproj}, any non-empty orthonormal set in $ M_2$ consisting of projections like $P_{\xi}$ has cardinality $1$ or $2$. So, it cannot be a basis for $M_2$, which has dimension $4$.

\begin{remarks} \label{3.5}
The basis given in Example \ref{anotherbasis} for $M_2$ can be referred to as {\it mixed type}, due to the fact that \\
$(a)~\langle P_{\xi^{(0)}}, P_{\xi^{(1)}}\rangle = 0$ \\
$(b)~\langle P_{\xi^{(j)}}, P_{\xi^{(k)}}\rangle = \frac{1}{2}$ for $j=0$ or $1$, $k=2$ or $3$\\
and \\
$(c)~\langle P_{\xi^{(2)}}, P_{\xi^{(3)}}\rangle = \frac{1}{2}.$
\end{remarks}

\begin{example}\label{twoequiangular}
Let $\lbrace \xi^{(j)},~0\leq j\leq 3\rbrace$ be as in Example \ref{anotherbasis}. Then one may verify that $S_1=\lbrace \xi^{(0)}, \xi^{(2)}, \xi^{(3)}\rbrace$ and $S_2=\lbrace \xi^{(1)}, \xi^{(2)}, \xi^{(3)}\rbrace$ are both equiangular in $\mathbb{C}^2$ with common angle $\frac{\pi}{4}$.
\end{example}

\begin{example}\label{lookforxi}
Let $\xi=\left( \begin{array}{c}
\alpha \\
\beta
\end{array} \right) \in \mathbb{C}^2$ with $\vert \alpha \vert^2+\vert \beta \vert ^2=1$. Then
\begin{eqnarray}
S(\xi)&=& \lbrace U_j \xi : 0\leq j\leq 3\rbrace \nonumber\\
&=&\Bigg \{ \left( \begin{array}{c}
\alpha \\
\beta
\end{array} \right), \left( \begin{array}{c}
\alpha \\
-\beta
\end{array} \right), \left( \begin{array}{c}
\beta \\
\alpha
\end{array} \right), \left( \begin{array}{c}
-i\beta \\
i\alpha
\end{array} \right) \Bigg \}.\nonumber
\end{eqnarray}

We look for $\xi$ for which $S(\xi)$ is an equiangular set (having four elements) in the sense of item (b) of Remark  \ref{admissrem}. This immediately rules out the cases $\beta=0, ~\alpha=0, ~\alpha=\beta$ and $\alpha=-\beta$. Further, it is enough to consider the case $\alpha>0$. So, to begin with, we may take $\alpha=cos \theta, ~\beta=sin \theta e^{i\phi}, ~0<\theta <\frac{\pi}{2},~-\pi <\phi \leq \pi$ and for $\theta=\frac{\pi}{4},~0 \neq \phi \neq \pi$.

We list different $\vert \langle x, y \rangle\vert$ for $x\neq y \in S(\xi)$: 
\begin{eqnarray}
\vert (\vert \alpha \vert^2-\vert \beta \vert ^2)\vert&=&\vert cos^2 \theta -sin^2 \theta \vert=\vert cos 2\theta \vert,\nonumber\\
\vert \overline{\alpha} \beta +\overline{\beta} \alpha \vert &=& \vert sin 2\theta cos \phi \vert,\nonumber
\end{eqnarray}
and 
\begin{eqnarray}
\vert \overline{\alpha} \beta -\overline{\beta }\alpha \vert &=& \vert sin 2\theta sin \phi \vert .\nonumber
\end{eqnarray}

So, $S(\xi)$ is equiangular if, and only if, 
\begin{center}
$\vert cos 2\theta \vert =\vert sin 2\theta cos \phi \vert =\vert sin 2\theta sin \phi \vert \neq 1,$
\end{center}
which holds if, and only if, $\phi$ is an odd multiple of $\frac{\pi}{4}$ and 
\begin{center}
$\vert cos 2\theta \vert =\frac{1}{\sqrt{2}}sin 2\theta,$
\end{center}
which in turn, holds if, and only if, 
\begin{center}
$\vert cos 2\theta \vert =\frac{1}{\sqrt{3}}$ and  $\phi =\pm \frac{\pi}{4}, \pm \frac{3\pi}{4},$
\end{center}
which holds if, and only if, both 
\begin{center}
$e^{i\phi}=\pm e^{\pm \frac{i\pi}{4}}$
\end{center}
and either 
\begin{center}
$(cos \theta, sin \theta )=(\sqrt{\frac{1}{2}(1+\frac{1}{\sqrt{3}})}, \sqrt{\frac{1}{2}(1-\frac{1}{\sqrt{3}})} ),$
\end{center}
or 
\begin{center}
$(cos \theta, sin \theta )=(\sqrt{\frac{1}{2}(1-\frac{1}{\sqrt{3}})}, \sqrt{\frac{1}{2}(1+\frac{1}{\sqrt{3}})} ).$
\end{center}

\noindent This gives eight possible solutions for $\xi=\left( \begin{array}{c}
\alpha \\
\beta
\end{array} \right)$ with $\alpha > 0$. However, the corresponding sets $S(\xi)$ are equivalent for some of them and essentially there are only two well-known solutions $\xi=\frac{1}{\sqrt{6}}\left( \begin{array}{c}
\sqrt{3+\sqrt{3}} \\
e^{\frac{i\pi}{4}}\sqrt{3-\sqrt{3}}
\end{array} \right)$ and $\xi=\frac{1}{\sqrt{6}}\left( \begin{array}{c}
-\sqrt{3-\sqrt{3}} \\
e^{\frac{i\pi}{4}}\sqrt{3+\sqrt{3}}
\end{array} \right)$.
\end{example}

We now summarize the findings of the discussion above.

\begin{theorem}\label{summth}
$(a)$ Let $S$ be a subset of unit vectors in $\mathbb{C}^2$ with $\left\vert{S}\right\vert >2$. Then $S$ is equiangular in $\mathbb{C}^2$ if, and only if, $\lbrace P_\xi:\xi \in S\rbrace$ is equiangular in $M_2$. In this case,  $\lbrace P_\xi:\xi \in S\rbrace$ is linearly independent in $M_2$.\\
$(b)~M_2$ has an equiangular basis consisting of rank-one projections whose sum is $2I$.
\end{theorem}
\begin{proof}
$(a)$ For the first part we have only to use Remark \ref{rankoneproj} and Remark \ref{admissrem}. For the second part, let $a\neq 1$ be the common value of $\lbrace \vert\langle \xi, \eta\rangle\vert^2 : \xi \neq \eta \in S \rbrace$. Let, if possible $\lbrace P_\xi : \xi \in S \rbrace$ be not linearly independent. Then there exists a finite non-empty subset $S_1$ of $S$ and a non-zero tuple $(\alpha_{\xi})_{\xi \in S_1}$ of scalars such that
$\sum_{\xi \in S_1}\alpha_{\xi}P_{\xi} =0$. Then for $\eta \in S_1$, one has: 
\begin{eqnarray}
 0&=&\sum_{\xi \in S_1}\alpha_{\xi}\langle P_{\xi}, P_{\eta}\rangle\nonumber\\
 &=&\alpha_{\eta}+\sum_{\eta \neq \xi \in S_1} \alpha_{\xi} a\nonumber\\
 &=&\alpha_{\eta}+a\sum_{\eta \neq \xi \in S_1} \alpha_{\xi}.\nonumber
 \end{eqnarray}
So, for $\eta_1 \neq \eta_2$ in $S_1$, one has:
 \begin{eqnarray}
 \alpha_{\eta_1}+a\alpha_{\eta_2}+a\sum_{\eta_1 \neq \xi \neq \eta_2, \xi \in S_1} \alpha_{\xi}=\alpha_{\eta_2}+a\alpha_{\eta_1}+a\sum_{\eta_1 \neq \xi \neq \eta_2, \xi \in S_1} \alpha_{\xi}.\nonumber
 \end{eqnarray}
This gives $(1-a)(\alpha_{\eta_1}-\alpha_{\eta_2})=0$, so that $\alpha_{\eta_1}=\alpha_{\eta_2}$. Thus, for any point $\eta \in S_1$, one has $\alpha_{\eta}(1+(\left\vert{S_1}\right\vert -1)a)=0$, so that $\alpha_{\eta}=0$ for $\eta \in S_1$, a contradiction. \\
(b) Follows by combining $(a)$ and Example \ref{lookforxi} above.
\end{proof}

Now the ground is set for the general concept of symmetric informationally complete positive operator valued measures, in short, SIC-POVM, or even SIC.

\subsection{Higher dimensions}\label{highdim}
Let $d\geq 2$ and $H=\mathbb{C}^d$, the usual inner product space with the standard basis of vectors $\vert j\rangle =e_j$, the tuple with $1$ at $j$th place and zero elsewhere, $j=0, 1, \cdots, d-1$. We look for analogues for general $d$ of different items in subsection \ref{twodim} above in a suitable order.

As before, we may consider $M_d(\mathbb{C})$, in short, $M_d$ as an innner product space with the inner product $\langle T, S \rangle=tr(T^*S)$ for $T$, $S$ in $M_d$. We note that the norm $\Vert . \Vert_2$ in $M_d$ is simply the Euclidean norm when $M_d$ is identified with $\mathbb{C}^{d^2}$, and it is the form for the inner product in terms of trace that turns out to be useful.

Definition of projections can be repeated verbatim for a unit vector $\xi$ and the same applies to expression for $tr(TP_{\xi})$ for an operator $T$ on $C^d$ to itself. Furthermore, an exact analogue of Remark \ref{rankoneproj} poses no problem to go over from $\mathbb{C}^2$ to $\mathbb{C}^d$. We observe that Definition \ref{defequiortho} and Remark \ref{admissrem} already hold for inner product spaces of higher dimensions, say $d$ and to avoid trivialities we take $\vert S \vert >d$. Also for the next paragraph in \ref{twodim} changing $\mathbb{C}^2$ to $\mathbb{C}^d$, $M_2$ to $M_d$ and $4=2^2$ to $d^2$ poses no problem at all. 

\subsubsection{Weyl-Heisenberg group or Schwinger basis}\label{whs}
Let $w_d=exp(\frac{2\pi i}{d})$ and $\tau_d=-exp(\frac{\pi i}{d})$. Let $X$ and $Z$ be operators on $H$ (introduced by Hermann Weyl in 1925) given by $X\vert j\rangle=\vert j+1\rangle$, $Z\vert j\rangle=w_d^j \vert j \rangle $ for $j\in \mathbb{Z}_d=\lbrace 0, 1, \cdots, d-1\rbrace$, the ring of integers with addition and multiplication modulo $d$. Then for $j, k \in \mathbb{Z}_d$, we note that $X^j Z^k=(w_d)^{-jk}Z^kX^j.$ Now, define the {\it Weyl-Heisenberg displacement operator} by $D_{jk}=(\tau_d)^{jk}X^jZ^k$. 

Then $D_{00}=I_H$ (the identity operator on $H$) and all $D_{jk}$'s are unitary operators with $tr(D_{jk})=0$ for $(j, k)\neq (0, 0)$. The product of two displacement operators is, up to a phase factor a third, in the sense that 
\begin{equation}
D_{jk}D_{\alpha \beta}=(\tau_d)^{k\alpha-\beta j}D_{j+\alpha, k+\beta}.\nonumber
\end{equation}

Thus, by allowing the generators to be multiplied by phase factors, we may ``define'' the so-called Weyl-Heisenberg group $\lbrace D_{jk} : j, k \in \mathbb{Z}_d \rbrace$. Next, $D_{jk}$ is a scalar multiple of $D_{dj, dk}$ as is clear from the definition for $j, k \in \mathbb{Z}_d$. So, in view of the product formula, one has  $tr(D_{jk}^*D_{\alpha \beta})=0$ for $(j, k)\neq (\alpha, \beta)$. This makes $\lbrace D_{jk}: j, k \in \mathbb{Z}_d\rbrace$ a linearly independent set of cardinality $d^2$ in $M_d$.

Note that for $d=2$, one has $D_{00}=I=P$, $D_{01}=A$, $D_{10}=U$, $D_{11}=L$.

 Once again, we summarize our observations below.

\begin{remarks}\label{sumhd}
$(a)$ An appropriate analogue for Example \ref{anotherbasis} appears to be that of the well-known {\it mutually unbiased bases} which exist, to begin with, when $d$ is prime. The unitary basis $\lbrace D_{jk}: j, k \in \mathbb{Z}_d\rbrace$ given as above can be expressed as a union of $d+1$ subsets say ${\cal U}_s$, $1\leq s \leq d+1$ with the property that the members of each ${\cal U}_s$ commute with each other. Furthermore, for $1\leq s \leq d+1$, we can choose a common orthonormal basis, say, $\lbrace \xi^{(s)}_j : 0\leq j \leq d-1 \rbrace $ for members of ${\cal U}_s$ that also satisfy $\vert\langle \xi^s_j, \xi^t_k \rangle\vert=\frac{1}{\sqrt{d}}$ for $s\neq t$, $j, k \in \mathbb{Z}_d$. The set 
\begin{center}
$\lbrace \xi^{(1)}_j : j \in \mathbb{Z}_d \rbrace \cup \lbrace \xi_j^s : 0 \neq j \in \mathbb{Z}_d, 2 \leq s \leq d+1 \rbrace$ 
\end{center}
works fine simply because 
\begin{center}
$P_{\xi_0^s} =I_d-\sum_{j=1}^{d-1} P_{\xi _j^s},~2 \leq s \leq d+1$
\end{center}
and for $1 \leq s \leq d+1$, each $U$ in ${\cal U}_s$ is of the form $\sum_{j=0}^{d-1} \alpha_j ^UP_{\xi^s_j}$ for eigenvalues $\alpha_j^U$ of $U$ with eigenvector $\xi_j^s$, $0 \leq j \leq d-1$.

$(b)$ For unitary bases ${\cal U}$ as in $(a)$, for the case of composite $d$'s, there is no guarantee for mutually unbiased bases. However, maximal commuting subsets of ${\cal U}$ can be combined and common orthonormal bases can be found for them. The number of rank one projections is much more than $d^2,$ in general. Various methods to reduce this number were given by a few authors including Chaturvedi, Mukunda and Simon \cite{cms}, and Shalaby and Vourdas \cite{sv1, sv2}. The second author in a joint work with Chaturvedi, Ghosh and Parthasarathy managed to reduce the number further by their method of optimal quantum tomography with constrained elementary measurements, in general and to $d^2$ in case $d=p^2$ or certain products of two distinct primes $p$ and $a$. The work is an augmented and refined combination of \cite{gs} and \cite{cgps}.

$(c)$ For $d$, a power of prime, there exist mutually unbiased bases for $\mathbb{C}^d$, where $\mathbb{C}^d$ is expressed as a tensor product and thus has another inner product and new unitary bases. Attempts have been made to obtain sets of projections of rank one (that will suffice for instance, in quantum tomography) for a general $d$ too, where $\mathbb{C}^d$ is endowed with new inner products and new unitary bases. 
\end{remarks}

 We now write an analogue of Remark \ref{3.5} based on $(a)$ of Remark \ref{sumhd}.

\begin{remarks}
Let $\lbrace P_{\xi_j^{(s)}} : 1 \leq s \leq d+1,~ 0 \leq j \leq d-1 \rbrace$ be as in Remark \ref{sumhd} $(a)$ above. Then, 
\begin{center}
$\langle P_{\xi_j^{(s)}},P_{\xi_k^{(t)}} \rangle=  \bigg \{ \begin{array}{ll}
0, &  for~s= t, j\neq k \\ 		
\frac{1}{d}, & for~ s \neq t
\end{array}.$ 
\end{center}
\end{remarks}
Finally, the proof of $(a)$ of Theorem \ref{summth} can be adapted to extend it to general $d \geq 2,$ the statement of which is as given below.

\begin{theorem}\label{genth}
Let $S$ be a subset of unit vectors in $\mathbb{C}^d$ with $\left\vert{S}\right\vert > 2$. Then $S$ is equiangular in $\mathbb{C}^d$ if, and only if, $\lbrace P_{\xi} : \xi \in S \rbrace$ is equiangular in $M_d$. In this case, $\lbrace P_{\xi} : \xi \in S \rbrace$ is linearly independent in $M_d$.
\end{theorem}

It is pertinent to point out that certain analogues of $(b)$ of Theorem \ref{summth} have been conjectured by Renes, Blume-Kohout, Scott and Caves \cite{Renes} and Zauner \cite{Zauner} in stronger forms or particular forms involving the Weyl-Heisenberg group (among other things). However, these have been proved only for some special cases. The rest of the article is devoted to an idea of that which brings in connections with Fibonacci numbers and Lucas numbers, in particular.

\subsection{Symmetric Informationally Complete Positive Operator Valued Measure (SIC-POVM) and Equiangular Lines}\label{sicpovm}
SIC-POVMs were introduced by Renes, Blume-Kohout, Scott and Caves \cite{Renes}. There, details of some basic examples, basic properties and the relationship to frames and spherical designs without knowing \cite{Zauner} and geometrical regular simplexes in $\mathbb{R}^{d^{2}}-1$ were given. We will confine over attention to the basics.

\subsubsection{The Concept}
As indicated in Subsection \ref{twodim} and Subsection \ref{highdim}, a set ${\cal S}=\lbrace S_t :1 \leq t \leq \tau \rbrace$ of positive operators on $\mathbb{C}^d$ can help to determine $T$ in the linear  span $E$ of ${\cal S}$ via $\lbrace tr(TS_t): 1 \leq t \leq \tau \rbrace$, called {\it measurements}; further for efficiency, $S_t$'s better be suitable positive multiples of projections of rank one, i.e. $P_{\xi_t}$'s with $\xi_t$'s being unit vectors. Any such ${\cal S}$ with $\sum_{t=1}^\tau S_t=I$ will be called a {\it positive operator valued measure}, in short POVM. At times, we will call $S=\lbrace \xi_t : 1 \leq t \leq \tau \rbrace$ also a POVM. Let us summarize our observations in what follows.

\begin{remarks}\label{sumsicpovm}
$(a)$ Clearly, $E=M_d$ if, and only if, ${\cal S}$ contains a linearly independent set of $d^2$ operators and, a fortiori, only if $\tau \geq d^2$. Such a POVM is called {\it IC-POVM}, where IC stands for {\it informationally complete}. 

$(b)$ It is clear from the paragraph just before Remark \ref{sumhd}, that 
\begin{center}
$\lbrace \frac{1}{d+1}P_{\xi_j^s} : 1\leq s \leq d+1, 0\leq j \leq d-1 \rbrace $
\end{center}
is an IC-POVM, though not linearly independent. The subset 
\begin{center}
$\lbrace \frac{1}{d+1}P_{\xi_j^s} : 1\leq s \leq d+1, 1\leq j \leq d-1 \rbrace \cup \lbrace \frac{1}{d+1} P_{\xi_0^1} \rbrace$
\end{center}
is linearly independent, it spans $M_d$ but the elements in it 
do not add to $I$.

$(c)$ What is appealing is the situation when there is {\it symmetry}, in the sense that $\tau=d^2$, $d{\cal S}=\{dS: S \in {\cal S}\}$ is equiangular with common angle $cos^{-1}(\frac{1}{d+1})$, or equivalently, $S=\lbrace \xi_t : 1\leq t \leq d^2 \rbrace$ is equiangular with common angle $cos^{-1}(\frac{1}{\sqrt{d+1}})$. Such an IC-POVM, if any, will be called a {\it symmetric informationally complete POVM}, in short SIC-POVM, or even SIC.

$(d)$ SIC-POVMs have been displayed, proved to exist for certain $d$'s, approximately determined by computer for more $d$'s and conjectured to exist for all $d$'s. The conjectures have become stronger or taken different forms over the time, but their proofs continue to be elusive. An idea follows in the order and form that we like.
\end{remarks}

\subsubsection{The Renes-Kohout-Scott-Caves Conjecture}\label{Renes}
The statement of the conjecture is this: There exists a unit vector $\xi \in \mathbb{C}^d$ such that 
\begin{center}
$S_{\xi}=\lbrace \xi_{jk}=D_{jk}\xi :0 \leq j, k \leq d-1 \rbrace$
\end{center}
is a SIC-POVM. 

We collect further observations in the following remark:

\begin{remarks}\label{sicpovm}
$(a)$ Such a SIC-POVM may be termed as {\it covariant} with the group $\mathbb{Z}_d \times \mathbb{Z}_d$ in the sense that $(j, k)\rightarrow D_{jk}$ is a projective or a ray representation of $\mathbb{Z}_d \times \mathbb{Z}_d$, (with addition modulo $d$) such that\\
($\alpha$) for each $(j^\prime, k^\prime) \in \mathbb{Z}_d \times \mathbb{Z}_d$, $S_{\xi}$ is invariant under $D_{j^\prime k^\prime}$, i.e., for $(j,k)$ in $Z_d \times Z_d$, there is $(j^{\prime \prime},k^{\prime \prime})$ in $Z_d  \times Z_d$ that satisfies  $D_{j^\prime k^\prime}\xi_{jk}=\xi_{j^{\prime \prime}k^{\prime \prime}}$ up to a phase, (group invariance) and\\
($\beta$) for any $(j,k)$, $(j^\prime, k^\prime) \in \mathbb{Z}_d \times \mathbb{Z}_d$ there exists $(j^{\prime \prime}, k^{\prime \prime}) \in \mathbb{Z}_d \times \mathbb{Z}_d$ such that $D_{j^{\prime \prime} k^{\prime \prime}}\xi_{jk}=\xi_{j^\prime k^\prime}$ up to a phase (transitivity).

$(b)$ Vector $\xi$ in $(a)$ above is called a {\it fiducial vector}. Clearly, in that case each $\xi_{jk}$ may be taken as a fiducial vector.

$(c)$ One may replace $\mathbb{Z}_d \times \mathbb{Z}_d$ by a group $G$ and $(j, k) \rightarrow D_{jk}$ by a projective representation of $G$ on $\mathbb{C}^d$ to generalize the concept.

$(d)$ The advantage in $(a)$ is: to check that $S_{\xi}$ is equiangular, it is enough to check that 
\begin{center}
$\lbrace \vert\langle \xi, D_{jk}\xi \rangle\vert, (j,k) \neq (0, 0) \rbrace$
\end{center}
is a singleton other than $\{1\}$. This may be verified by Example 3.7.

It is instructive to look at the details for $d=4$ as in \cite{Renes}, for instance and note that it involves $\sqrt{5}$ instead of $\sqrt{3}$ for $d=2$, a fact of importance for certain higher $d$'s as elaborated in the next two subsections.
\end{remarks}

\subsection{SIC-POVM and Algebraic Number Theory}\label{sicant}
This subsection is essentially a summary of a some extracts from \cite{NewaddedApple}, \cite{Applebytower}, \cite{Applebyray} and \cite{Bengtsson} in a form and order that we have preferred. The purpose is to display some deep relationships between SIC-POVMs and algebraic number theory.

It is clear from Example 3.7 and the discussion in the previous subsection, that SIC vectors are likely to be expressed in terms of algebraic numbers. 

In view of the comments made in Subsection \ref{highdim} and Remark \ref{sicpovm} of the previous section, for a rank one projection $P_{\xi}$ and for $j, k \in \mathbb{Z}_d$, $tr((P_{\xi} D_{jk})=\langle \xi, D_{jk} \xi \rangle$. So by Remark 3.13, it follows that $P_{\xi}$ is a fiducial projector or projection (i.e., $\xi$ is a fiducial vector) if, and only if, for $j, k \in \mathbb{Z}_d$, $(j,k) \neq (0,0)$ one has, 
\begin{center}
$tr(P_{\xi}D_{jk})=\frac{e^{i\theta_{jk}}}{\sqrt{d+1}}$
\end{center}
for some phase factor $e^{i\theta_{jk}}$. Next, in view of the discussion in item \ref{whs}, 
\begin{center}
$P_{\xi}=\frac{1}{d}\sum_{j,k \in \mathbb{Z}_d}\frac{e^{i\theta_{jk}}}{\sqrt{d+1}}D_{jk}^*.$
\end{center}
Thus $P_{\xi}$ can be constructed from phase factor $e^{i\theta_{jk}}$'s. 

For $d=4$, the phase factors turn out to be $\pm u, \pm\frac{1}{u}, -1$ with $u=\frac{\sqrt{5}-1}{2\sqrt{2}}+i \frac{\sqrt{\sqrt{5}+1}}{2}$. The smallest number field $Q(u)$ containing $Q$, the field of rationals and $u,$ can be seen to be $Q(\sqrt{5}, \sqrt{2}, i\sqrt{\sqrt{5}+1})$. Its minimal polynomial is $t ^8-2t^6-2t ^4-2t^2+1$. So $Q(u)$ is a vector space over $Q$ of dimension $8$. We refer the reader to Bengtsson \cite{Bengtsson} for all the details, where an indication for larger $d$'s is also provided. 
 
Let $d\geq 4$ and $D$ be the square free part of $(d-3)(d+1)=(d-1)^2-4$. Let $P_{\xi}$ be a SIC-POVM fiducial projector, if any, and let $\mathbb{E}=Q(P_{\xi}, \tau)=\tau_d$ be the field generated over the rationals by the standard matrix basis elements of $P_{\xi}$ together with $\tau_d$. Appleby, Flammina, McConnell and Yard present various facts observed by them in known cases \cite{Applebyray}. We state only a few of those, to give a feel of the relationship between SIC-POVMs and algebraic number theory.

$(a)~\mathbb{E}$ is an extension of $\mathbb{K}=Q(\sqrt{D})$.\\
$(b)~Gal(\mathbb{E}/\mathbb{K})$ is abelian.

To get an idea of sparseness, we note that
\begin{eqnarray}
d&=& 7 , 35, 199, 1155, 6727, 39203, 228487, \cdots \textrm{for} ~D=2,\nonumber\\
d&=& 15, 53, 195, 725, 2703, 10085, \cdots \textrm{for}~ D=3,\nonumber\\
d&=& 4, 8, 19, 48, 124, 323, 844,  \cdots \textrm{for}~ D=5,\nonumber
\end{eqnarray}

But the fact remains that while we keep adding more known cases possessing SIC-POVMs but we keep shifting from one conjecture to another. 

Finally, the next subsection is devoted to the case $D=5$ and involves Fibonacci matrix and Lucas numbers, thereby taking us back to the objective of this survey. 

\subsection{Fibonacci-Lucas SIC-POVMs}\label{flsic}
It may be remarked that the development by Grassl and Scott \cite{ScottGrassl} motivated us to write this article, in fact. They presented a conjectured family of symmetric SIC-POVMs which have an additional symmetry group, whose size grows with the dimension. It may be mentioned here that Scott is a pioneer both in giving exact solutions and in providing numerical solutions to a high level of precision or approximation and later collaborated with Grassl. The solutions spread over many publications by different authors are available together at the following website (as reported in \cite{FuchStacy}): http://www.physics.umb.edu/Research/QBism. 

\subsubsection{Additional restrictions and Zauner's Conjecture} 
For finding a fiducial vector $\xi$, when one employs an exact search or uses a method involving numerical bounds, it helps to know if $\xi$ can be chosen with more restrictions. A beginning was made by the originator Zauner himself, who observed that, in known cases $\xi$ is in a particular eigenspace of an order three unitary and turned it into a conjecture now called {\it Zauner's Conjecture}.

We summarize further observations in the following remark:
\begin{remarks}
$(a)$ We start with the case of odd $d$ for simpler introduction. Let $F$ be an invertible $2\times 2$ matrix over $\mathbb{Z}_d$ with determinant one. This induces a unitary symmetry on the Weyl-Heisenberg group (ignoring phase factors) in the sense that there is a unitary $U_F$ on $M_d$ such that
\begin{equation}
U_F D_{jk}U_F^*=D_{j^\prime k^\prime}\nonumber
\end{equation}
with 
\begin{equation}
\left( \begin{array}{c}
j^\prime \\
k^\prime
\end{array} \right)=F\left( \begin{array}{c}
j \\
k
\end{array} \right)\nonumber
\end{equation}
for $j, k \in \mathbb{Z}_d$.
This $U_F$ is unique upto an overall phase.

$(b)$ If we take $F$ in $(a)$ to have determinant $-1$ instead, then $U_F$ will be anti-unitary in the sense that $\langle U_Fx, U_Fy\rangle=\langle y, x \rangle$ for $x, y \in \mathbb{C}^d$.

$(c)$ If we take $d$ even in $(a)$ and $(b)$ above, then $F$ will have to be considered in $\mathbb{Z}_{d^\prime}$ with $d^{\prime}=2d$.

$(d)$ Zauner takes $F$ to be 
\begin{equation}
F_Z=\left( \begin{array}{cc}
0 & -1 \\
1 & -1
\end{array} \right).\nonumber
\end{equation}
In fact, for $d  \neq 3 ~(mod~ 9),$ every canonical order three unitary is equivalent to $F_Z$, though for $d=9l+3$, $l\geq 1$, $F_a=\left( \begin{array}{cc}
1 & 3 \\
2l & -2
\end{array} \right)$ is also a canonical order three unitary which is not conjugate to $F_Z$.
\end{remarks}

$(e)$ Grassl and Scott \cite{ScottGrassl} consider the variant $F_f=\left( \begin{array}{cc}
0 & 1 \\
1 & 1
\end{array} \right)$ of the Fibonacci matrix $\left( \begin{array}{cc}
1 & 1 \\
 1 & 0
\end{array} \right)$. Just as in the Introductory section, $F_f$'s eigenvalues are $\phi=\frac{1+\sqrt{5}}{2}$ and $\psi=-\frac{1}{\phi}=\frac{1-\sqrt{5}}{2}$ with $\eta=\left( \begin{array}{c}
1 \\
\phi
\end{array} \right)$ and $\xi=\left( \begin{array}{c}
1 \\
\psi
\end{array} \right)$ as corresponding eigenvectors. 

Fibonacci numbers $F_n$'s are given by
\begin{eqnarray}
\left( \begin{array}{c}
F_n \\
F_{n+1}
\end{array} \right)&=& F_f^n\left( \begin{array}{c}
0 \\
1
\end{array} \right)\nonumber\\
&=&F_f^n(\frac{1}{\phi -\psi}(\eta -\xi))\nonumber\\
&=&\frac{1}{\sqrt{5}}(\phi^n \eta -\psi^n \xi)\nonumber\\
&=&\frac{1}{\sqrt{5}}\left( \begin{array}{c}
\phi ^n-\psi ^n \\
\phi^{n+1}-\psi^{n+1}
\end{array} \right)\nonumber,
\end{eqnarray}

so that $F_n=\frac{\phi ^n-\psi^n}{\phi- \psi}$ for each $n$, while the Lucas numbers $L_n$'s may be determined as follows: 

For each $n$,
\begin{eqnarray}
\left( \begin{array}{c}
L_n \\
L_{n+1}
\end{array} \right)&=& F_f^n\left( \begin{array}{c}
L_0 \\
L_1
\end{array} \right)\nonumber\\
&=& F_f^n\left( \begin{array}{c}
2 \\
1
\end{array} \right)\nonumber\\
&=&F_f^n(\eta + \xi)\nonumber\\
&=&\phi^n \eta+\psi^n \xi \nonumber\\
&=&\left( \begin{array}{c}
\phi ^n+\psi ^n \\
\phi^{n+1}+\psi^{n+1}
\end{array} \right).\nonumber
\end{eqnarray}
So, $L_n=\phi^n+\psi^n$ for all $n$.

Finally we see that $F_f^n$ is simply $\left( \begin{array}{cc}
F_{n-1} & F_n \\
F_n & F_{n+1}
\end{array} \right)$ for all $n>0$.

\subsubsection{The Scott and Grassl Conjecture}
The conjecture is the statement: For the infinite sequence of dimensions $d_k=\phi ^{2k}+\phi^{-2k}+1$, where $\phi=\frac{\sqrt{5}+1}{2}$, $k\geq 1$, there exists a Weyl-Heisenberg covariant SIC-POVM that has an additional anti-unitary symmetry of order $6k$ given by the Fibonacci matrix
\[F_f=\left( \begin{array}{cc}
0 & 1 \\
1 & 1
\end{array} \right).\]
We term such a set of vectors a {\it Fibonacci-Lucas SIC-POVM}.

Let us record some facts proved in \cite{ScottGrassl}.

\begin{remarks}
$(a)$ The sequence $d_k$ obeys the linear recurrence relation
\begin{equation}
d_{k+3} = 4d_{k+2}-4d_{k+1} + d_k.\nonumber
\end{equation}

$(b)$ Considering $modulo~ 3$, the sequence $\tilde{d_k} =d_k ~(mod ~3) $ has period four and is given by $1, 2, 1, 0, 1, 2, 1, 0, \cdots$, (for $k=1, 2, 3, \cdots$). This implies that $d_k$ is divisible by $3$ if, and only if, $k=4l$. Then $d_{4l}=3~ (mod ~ 9)$.

$(c)$ The square-free part of $(d_k+1)(d_k-3)$ equals $5$.

$(d)$ We consider $F_f$ and its powers with entries in $\mathbb{Z}_d$ with $d=d_k$. Then $F_f$ has order $6k$.
\end{remarks}

Based on these observations, these authors presented the conjectured family of symmetric SIC-POVMs which have an additional symmetry group (generated by $F_f$), whose size grows with the dimension. While the symmetry group is related to Fibonacci numbers, the dimension is related to Lucas numbers. 

The conjecture is demonstrated by exact solutions for dimensions $d=4, 8, 19, 48, 124$ and $323$ and a numerical solution for $d=844,$ as well.

Details of exact solutions and numerical solutions are provided in \cite{ScottGrassl} which can be found in the paper and online at http://sicpovm.marks-grassl.de. 

\subsection{Concluding Remarks}
Fibonacci numbers are everywhere and so are polynomials as we saw above, for instance. The second author and Somshubhro Bandyopadhyay \cite{Som} have given polynomial representations of quantum entanglements. Let us also point to the fact that Fibonacci numbers enter the scene in polynomial representation of quantum entanglement of Resonance Valence Bond states as a small part of an ongoing work of the second author with Aditi Sen De and Ujjwal Sen.

\section{Acknowledgements}
A preliminary version of this exposition formed parts of the talks at the ``International Conference on Linear Algebra and its Applications''-ICLAA 2017 organised at the Manipal University. The authors discussed the plan for the present survey article at the conference. They thank the organizers, in particular K. Manjunatha Prasad for the opportunity. Ajit Iqbal Singh also thanks the Indian National Science Academy for continuous support and the Indian Statistical Institute, New Delhi for excellent research facilities.

\end{document}